\documentclass[10pt, a4paper]{amsart}
\usepackage{geometry} 
\usepackage{url}
\usepackage{amsthm}
\usepackage{amssymb}
\usepackage{latexsym}
\usepackage{amsfonts}
\usepackage{amsmath}
\usepackage{amstext}
\usepackage{accents}
\usepackage{cancel}
\usepackage{color}
\usepackage{soul}
\usepackage{enumerate}
\usepackage{eucal}
\usepackage{amscd}
\usepackage{hyperref}
\usepackage{graphicx}
\usepackage{verbatim}
\usepackage{caption}
\usepackage{enumitem}

\usepackage[usenames]{xcolor}
\setstcolor{red}

\definecolor{greenish}{RGB}{50,160,0}

\newtheorem{theorem}{Theorem}[section]

\newtheorem{lemma}[theorem]{Lemma}

\theoremstyle{definition}
\newtheorem{definition}[theorem]{Definition}

\theoremstyle{remark}
\newtheorem{remark}[theorem]{Remark}

\numberwithin{equation}{section}

\newcommand{\abs}[1]{\left|#1\right|}

\newcommand{\Div}{\textup{div}}

\newcommand{\id}{\textup{id}}

\newcommand{\R}{\mathbb{R}}

\begin{document}
\title[IMCF solitons with cylindrical ends]{Rotational symmetry of self-expanders to the inverse mean curvature flow with cylindrical ends}

\author[Gregory Drugan]{Gregory Drugan}
\address{Oregon Episcopal School, 6300 SW Nicol Road, Portland, OR 97223, USA}
\email{drugan.math@gmail.com}
\author[Frederick Fong]{Frederick Tsz-Ho Fong}
\address{Department of Mathematics, Hong Kong University of Science and Technology, Clear Water Bay, Kowloon, Hong Kong}
\email{frederick.fong@ust.hk}
\author[Hojoo Lee]{Hojoo Lee}
\address{Korea Institute for Advanced Study, Hoegiro 85, Dongdaemun-gu, Seoul 02455, Korea}
\email{momentmaplee@gmail.com}

\begin{abstract}
We show that any complete, immersed self-expander to the inverse mean curvature 
flow, which has one end asymptotic to a cylinder, or has two ends asymptotic to two coaxial cylinders, must be rotationally symmetric.
\end{abstract}

 
 \maketitle
 


 \section{Introduction}

In this paper, we establish the following rigidity result for non-compact, complete self-expanders to the inverse mean curvature flow.

\begin{theorem}
\label{thm:main:intro}
Let $F:\Sigma^{n \geq 2} \to \mathbb{R}^{n+1}$ be a complete, immersed self-expander to the inverse mean curvature flow, and suppose either:
\begin{itemize}
\item $F(\Sigma)$ has only one end which is asymptotic, in the sense of Defiition~\ref{asmp}, to a round cylinder; or
\item $F(\Sigma)$ has only two ends which are asymptotic, in the sense of Defiition~\ref{asmp}, to two coaxial round cylinders.
\end{itemize}
Then, $F(\Sigma)$ must be rotationally symmetric with respect to the axis of the asymptotic cylinder(s).
\end{theorem}

An immersion $F:{\Sigma}^{n} \to \mathbb{R}^{n+1}$ from an $n$-dimensional orientable manifold $\Sigma$ into $\mathbb{R}^{n+1}$ is a \emph{self-expander} to the inverse mean curvature flow (IMCF) if there exists a constant $C>0$ so that 
\begin{equation}
\label{eq:IMCF_Expander}
 -\frac{1}{H} = C \langle F, \nu \rangle,
\end{equation}
where $H=-\Div_{{}_{\Sigma}} \nu$ is the scalar mean curvature induced by the unit normal vector field $\nu$ on $\Sigma$. The mean curvature $H$ is assumed to be non-zero everywhere on a self-expander to IMCF. In particular, the support function $\langle F, \nu \rangle$ is non-zero everywhere and equation~(\ref{eq:IMCF_Expander}) is equivalent to the nonlinear elliptic equation
\begin{equation}
\langle \Delta_\Sigma F, \nu \rangle = -\frac{1}{C \langle F, \nu \rangle},
\end{equation}
where $\Delta_\Sigma$ is the Laplace-Beltrami operator induced on $\Sigma$ by the immersion $F$.

The round spheres ($\mathbb{S}^n$) and round cylinders ($\mathbb{S}^{n-1} \times \mathbb{R}$) are examples of self-expanders to IMCF in ${\mathbb{R}}^{n+1}$. Gerhardt \cite{G1990} and Urbas \cite{U1990} showed that compact, star-shaped, initial surfaces with strictly positive mean curvature converge under IMCF, after suitable rescaling, to a round sphere. Recently, the first named author, G. Wheeler, and the third named author \cite{DLW2015} proved that round spheres are the only closed self-expanders. While the round spheres are rigid among the compact self-expanders, there is no such rigidity for the round cylinders. Even in the class of hypersurfaces with rotational symmetry, there are known examples of non-compact, complete self-expanders different than the round cylinders. G. Huisken and T. Ilmanen \cite{HI1997} constructed a complete, rotationally symmetric, self-expander with one end which is asymptotic to a cylinder (see \cite[p. 10]{HI1997} for a graphical plot). In \cite{DLW2015}, examples of complete, rotationally symmetric, expanding topological cylinders with two different ends, called \textit{infinite bottles}, were constructed (see Figure \ref{fig_immersed}). In view of these known examples of rotationally symmetric self-expanders to IMCF with asymptotically cylindrical ends, the conclusion of Theorem~\ref{thm:main:intro} is essentially sharp. This is in contrast to L. Wang's uniqueness result \cite{Wang2016} for \emph{self-shrinkers} to the \emph{mean curvature flow} with asymptotically cylindrical ends.

 \begin{figure}
 \centering
 \includegraphics[height=2.65cm]{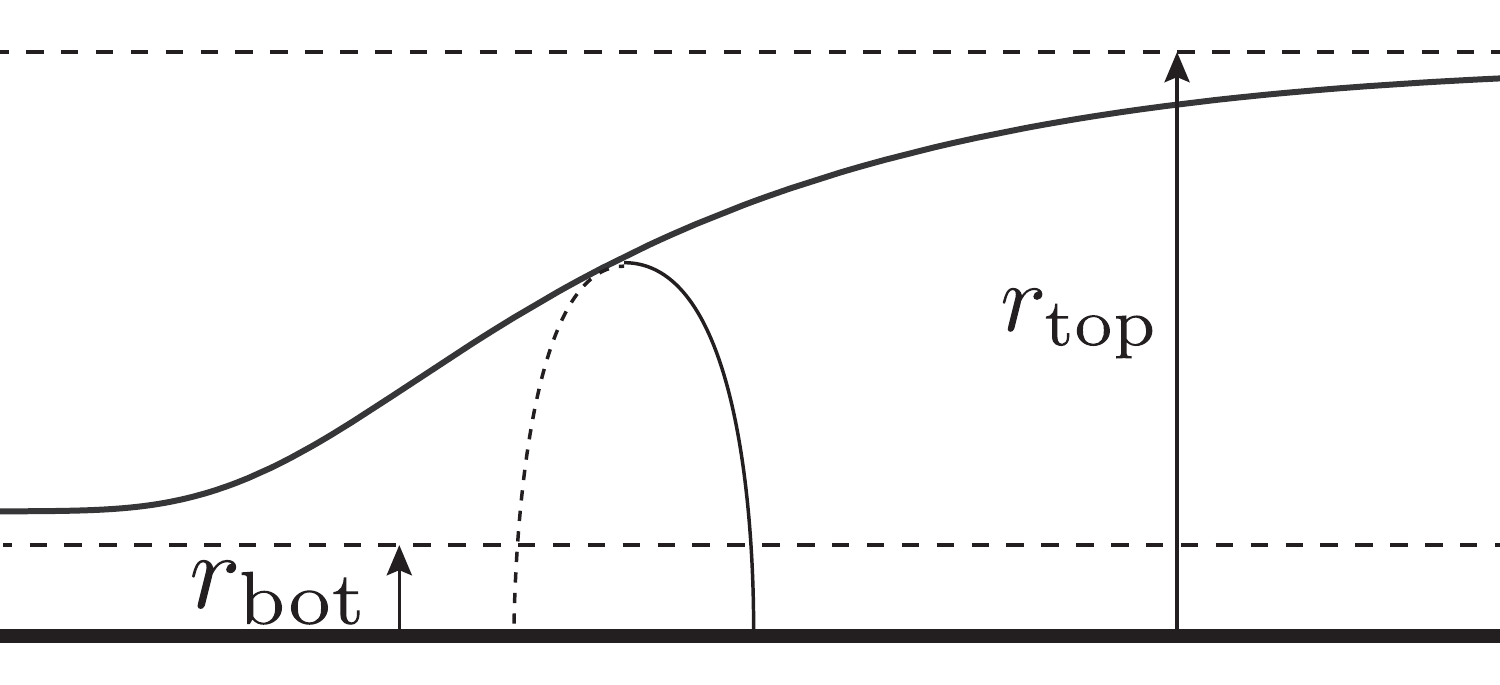}
 \caption{A numerical approximation of the part of a curve whose rotation about the horizontal axis is a self-expanding \textit{infinite bottle} in ${\mathbb{R}}^{3}$.}
 \label{fig_immersed}
 \end{figure}

Our proof of Theorem~\ref{thm:main:intro} is essentially motivated by the role of the stability operator in the classical theory of hypersurfaces with constant mean curvature. For instance, 
L. J. Al\'{i}as,  R. L\'{o}pez, and B. Palmer \cite{ALP1999} exploited the Jacobi field induced by the rotation vector field to show that 
the only stable constant mean curvature surfaces of genus zero in ${\mathbb{R}}^{3}$ with circular boundary are the spherical caps or the flat disks.
In our situation, we construct two independent kernel functions of the linearized operator of the soliton equation, induced by rotations and dilations, respectively. This is done in Section~\ref{sec:kernel}. Then, we utilize these kernel functions and the asymptotic behavior of our solitons in a maximum principle argument in Section~\ref{sec:main} to establish the rotational rigidity. Our proof is also motivated by recent results on the rigidity of various geometric solitons, such as 
steady Ricci solitons by S. Brendle \cite{Brendle2013, Brendle2014}, expanding Ricci solitons by O. Chodosh \cite{Ch2014}, 
K\"{a}hler-Ricci solitons by O. Chodosh and the second named author \cite{CF2016}, shrinking Ricci solitons by B. Kotschwar 
and L. Wang \cite{KL2015}, shrinkers to the mean curvature flow by M. Rimoldi \cite{Rim2014}, and  translators to the mean curvature flow by R. Haslhofer \cite{Haslhofer2015}. Finally, we mention the following references for background on the inverse mean curvature flow and its intriguing geometric applications: Huisken-Ilmanen \cite{HI1997}, Huisken-Ilmanen \cite{HI2001}, Bray-Neves \cite{BN2004}, Akutagawa-Neves \cite{AN2007}, Kwong-Miao \cite{KM2014}, Li-Wei  \cite{LW2015}, Brendle-Hung-Wang \cite{BHT2016}, Guo-Li-Wu \cite{GLW2016}, Lambert-Scheuer \cite{LS2016}, and Allen \cite{Al2016}.


 \section{Kernel of the linearized operator of  the soliton equation}
 \label{sec:kernel}

Let $F:\Sigma^n \to \mathbb{R}^{n+1}$ be an immersion from an $n$-dimensional manifold $\Sigma$ into $\mathbb{R}^{n+1}$, and let $\nu$ denote a continuous choice of unit normal along the hypersurface $F(\Sigma)$. Also, let $R$ be a rotation vector field about the axis of a round cylinder. We will first establish in Lemma \ref{kernel} that the inner products $\langle R, \nu \rangle$ and $\langle F, \nu\rangle$ both satisfy the same elliptic equation. The main result of this paper will then be proved by using the strong maximum principle.

To begin, we observe that given any self-expander $F$ to the IMCF, the induced one-parameter family $\left\{F_t := e^{Ct}F \right\}_{t \in \mathbb{R}}$ satisfies  IMCF. Indeed, we have  
\[
\left(\frac{\partial F_t}{\partial t}\right)^{\perp} = C e^{Ct}  \langle F, \nu \rangle \nu= - \frac{ 1 }{ e^{-Ct} H} \nu  = -\frac{1}{H_t} {\nu}_t. 
\]
Here, $H_t$ and $\nu_t$ denote the mean curvature and unit normal of the immersion $F_t$, respectively, and $\perp$ denotes the normal projection.

Noting that the mean curvature and unit normal are both invariant under re-parametrization, we can re-parametrize $\Sigma$ by some (time-dependent) tangential diffeomorphism $\Phi_t$ with $\Phi_0 = \id$ such that $F_t \circ \Phi_t$ satisfies the IMCF:
\[\frac{\partial}{\partial t}\left(F_t \circ \Phi_t\right) = -\frac{1}{H_t \circ \Phi_t}(\nu_t \circ \Phi_t).\]
With a little abuse of notations, we replace $F_t$ by $F_t \circ \Phi_t$ so that $F_t$ satisfies
\begin{equation}
\label{eq:IMCF}
\frac{\partial F_t}{\partial t} = -\frac{1}{H_t}\nu_t.	
\end{equation}

To investigate the self-shrinkers to the mean curvature flow, Colding and Minicozzi  \cite{CM2012} studied the linearized operator of the self-shrinker equation.
Here, we construct two kernel functions of the linearized operator of the self-expander equation to the inverse mean curvature flow. We will make use of the following well-known evolution equations:

\begin{lemma}[c.f. \cite{HP1996}]
\label{lma:Evolutions}
Let $F_{t}(p): \Sigma \times (-\varepsilon, \varepsilon) \to \R^{n+1}$ be a smooth one-parameter family of immersions in $\mathbb{R}^{n+1}$ satisfying the normal evolution 
\begin{equation}  \label{normalevolution}
 \frac{\partial {F}_{t}}{\partial t} = {\varphi}_{t} \, {\nu}_{t},
\end{equation}
where ${\varphi}_{t}$ is a smooth function and ${\nu}_{t}$ is the unit normal. We have the evolution equations
\begin{enumerate}
\setlength\itemsep{0.5em}
	\item $\displaystyle{\frac{\partial  }{\partial t} H_{t}  = \left(\Delta_\Sigma + \abs{A}^2\right) {\varphi}_{t}}$,
	\item $\displaystyle{\frac{\partial  }{\partial t} {\nu}_{t}  = -\nabla_\Sigma {\varphi}_{t}}$,
 \end{enumerate}
 where we adopt the notations $\Delta_\Sigma = \Delta_{F_{t}}$, $\nabla_\Sigma=\nabla_{F_{t}}$, and the second fundamental form $A=A_{F_{t}}$.
\end{lemma}
\begin{proof}
See e.g. \cite[p. 52--53]{HP1996}.
\end{proof}

Next we consider two kernel functions associated to the self-expander, namely the \emph{rotation function} and the \emph{support function}. Let $R$ be a rotation vector field in $\mathbb{R}^{n+1}$ about an axis. For instance in $\mathbb{R}^3$, a rotation vector field about the $z$-axis is given by $R(x,y,z) = (-y,x,0)$. The rotation function is given by  $\langle R(F), \nu \rangle$ (or $\langle R, \nu\rangle$ in short) where $F$ is the immersion of the self-expander, and $\nu$ is the unit normal. The support function is given by $\langle F, \nu\rangle$.
 
\begin{lemma}[\textbf{Linearized operator of the soliton equation}]  \label{kernel}
Given an expander $F:\Sigma \to \mathbb{R}^{n+1}$ satisfying \eqref{eq:IMCF_Expander} for some $C>0$, we associate it with the second-order elliptic operator 
\begin{equation} \label{SOLoperator}
 {\mathcal{L}}_{\Sigma}  {\varphi} := \Delta_\Sigma  {\varphi}+ CH^2\left\langle F, \nabla_{\Sigma}  {\varphi} \right\rangle + \left(\abs{A}^2 - CH^2\right)  {\varphi}.
\end{equation}
Then, $\langle R, \nu\rangle$ and $\langle F, \nu\rangle$ are in the kernel of $\mathcal{L}_\Sigma$:
\begin{itemize}
\item \label{lma:Lf_R} \textbf{(Kernel induced by rotational invariance of the soliton equation)} Given a rotation vector field $R$ in $\mathbb{R}^{n+1}$, the rotation function $ \langle R, \nu \rangle: \Sigma \to \mathbb{R}$ satisfies  
\begin{equation}
\label{eq:Lf_R}
{\mathcal{L}}_{\Sigma} \langle R, \nu \rangle=0.
\end{equation}

\item \label{lma:L(1/H)} \textbf{(Kernel induced by homothetical invariance of the soliton equation)}  The support function  $\langle F, \nu \rangle: \Sigma \to \mathbb{R}$ satisfies 
\begin{equation}
\label{eq:L(1/H)}
{\mathcal{L}}_{\Sigma}  \langle F, \nu \rangle =0.
\end{equation}

\end{itemize}
\end{lemma}

 \begin{proof} Our variational proofs illustrate that $\mathcal{L}_{\Sigma}$ is the linearized operator of the soliton equation. (For other approaches, see Remark \ref{general equality}). We shall prove \eqref{eq:Lf_R} and \eqref{eq:L(1/H)} on a sufficiently small neighborhood of any point on the surface. Observe that, for any normal 
deformation $\frac{\partial {F}_{t}}{\partial t} = {\varphi}_{t} \, {\nu}_{t}$ with $F_0=F$, letting ${\varphi}={\varphi}_{0}$, equation (\ref{normalevolution}) and identity (2) in Lemma \ref{lma:Evolutions} yield  
\begin{equation} \label{variation of support function}
   \left.\frac{\partial  }{\partial t}   \langle F_t, \nu_t \rangle \right|_{t=0}  =  \varphi -  \langle F,  \nabla_\Sigma  \varphi \rangle.
\end{equation}   

To prove \eqref{eq:Lf_R}, we prepare the one-parameter family $\{ {\psi}_{t} :  {\mathbb{R}}^{n+1} \to {\mathbb{R}}^{n+1} \}_{t \in \mathbb{R}}$ of the unit speed rotations generated by $R$. For instance in $\mathbb{R}^3$, if $R = (-y,x,0)$ then we have:
\[
   \psi_{t} (x, y, z) = \left( \, \left(\cos t \right) x -  \left(\sin t \right) y, \, \left(\sin t \right) x + \left(\cos t \right) y, \, z \, \right), 
\]
with the initial velocity $R(x,y,z)   =(-y,x,0) =     \left. \frac{\partial  }{\partial t}   \left(  \psi_{t}  (x,y,z) \right) \right|_{t=0}$.
Introducing $\widehat{R}:= \frac{\partial   \psi_{t} }{\partial t} $ and taking the one-parameter family $\left\{{\mathcal{F}}_t := {\psi}_{t} \circ F \right\}_{t \in \mathbb{R}}$  of immersions, 
we obtain
\[
 \frac{\partial  }{\partial t}  {\mathcal{F}}_{t}   = { \left( \widehat{R} \circ F  \right) }^{\top} +  { \left( \widehat{R} \circ  F  \right) }^{\bot}  = { \left( \widehat{R} \circ  F  \right) }^{\top} +   \langle \widehat{R}  \circ F , {\nu}_{t}  \rangle  \, {\nu}_{t}.    
\]
In a space-time neighborhood, one can always find some (time-dependent) tangential diffeomorphism ${\Phi}_t$ with $\Phi_0 = \id$ so 
that the reparametrization $F_t :={\mathcal{F}}_{t} \circ {\Phi}_t$ satisfies the normal evolution 
\[
\frac{\partial {F}_{t}}{\partial t} = {\varphi}_{t} \, {\nu}_{t}  
\]
with ${\varphi}_{t}(p) = \langle \widehat{R}  \circ F,  {\nu}_{t} \rangle$.
Due to the rotational invariance of the soliton equation, we also meet
\[
0=   \frac{1}{H_t} + C \langle {F}_{t}, \nu_{t}  \rangle.
\]
 Notice that  ${\varphi} = {\varphi}_{0}=  \langle  {R} ,  {\nu} \rangle $ on $\Sigma$, so identity (1) of Lemma \ref{lma:Evolutions} and  (\ref{variation of support function}) imply
\[
  0 =  \left.\frac{\partial  }{\partial t}  \left(  \, \frac{1}{H_t} + C \langle {F}_{t}, \nu_{t}  \rangle \right)   \right|_{t=0} 
= - \frac{1}{H^2}  \left(\Delta_\Sigma  \varphi  + \abs{A}^2  \varphi  \right) +  C  \left( \varphi -  \langle F,  \nabla_\Sigma  \varphi \rangle \right),
\]
 which is equivalent to  ${\mathcal{L}}_{\Sigma} \varphi=0$ on the initial surface $\Sigma$. 

Next we prove \eqref{eq:L(1/H)}. Let $\mathbf{X}$ denote the position vector field in ${\mathbb{R}}^{n+1}$. 
We prepare the one-parameter family $\{ {\psi}_{t} :  {\mathbb{R}}^{n+1} \to {\mathbb{R}}^{n+1} \}_{t \in \mathbb{R}}$ of homotheties  
\[
   \psi_{t} (\mathbf{X}) = e^{t}\mathbf{X},
\]
with the initial velocity vector field  $\left. \frac{\partial  }{\partial t}   \left(  \psi_{t} \circ \mathbf{X}  \right)  \right|_{t=0}=  \mathbf{X}$.  
Introducing $\widehat{\mathbf{X}}:= \frac{\partial   \psi_{t} }{\partial t} $ and taking the one-parameter family $\left\{{\mathcal{F}}_t := {\psi}_{t} \circ F \right\}_{t \in \mathbb{R}}$ of the immersions, we obtain
\[
 \frac{\partial  }{\partial t}  {\mathcal{F}}_{t}   = { \left(  \widehat{\mathbf{X}} \circ  F  \right) }^{\top} +   \langle {\widehat{\mathbf{X}}}  \circ F, {\nu}_{t}  \rangle  \, {\nu}_{t}. 
\]
In a space-time neighborhood, we can find a time-dependent tangential diffeomorphism ${\Phi}_t$ with $\Phi_0 = \id$ so 
that $F_t :={\mathcal{F}}_{t} \circ {\Phi}_t$ satisfies the normal evolution 
$\frac{\partial {F}_{t}}{\partial t} = {\varphi}_{t} \, {\nu}_{t}$ 
with ${\varphi}_{t} = \langle  {\widehat{\mathbf{X}}}  \circ F,  {\nu}_{t} \rangle$. Due to the homothetical invariance of the soliton equation, we have
\[
 -  \frac{1}{H_t} =  C \langle {F}_{t}, \nu_{t}  \rangle.
\]
 Note that  ${\varphi} = {\varphi}_{0}=  \langle  F,  {\nu} \rangle $ on $\Sigma$. We finally have  
\[
  0 =  \left.\frac{\partial  }{\partial t}  \left(  \, \frac{1}{H_t} + C \langle {F}_{t}, \nu_{t}  \rangle \right)   \right|_{t=0} 
= - \frac{1}{H^2}  \left(\Delta_\Sigma  \varphi  + \abs{A}^2  \varphi  \right) +  C  \left( \varphi -  \langle F,  \nabla_\Sigma  \varphi \rangle \right).
\]
 \end{proof}

\begin{remark} \label{general equality} The  identities \eqref{eq:Lf_R} and \eqref{eq:L(1/H)} in Lemma \ref{kernel} may be deduced in several ways. For instance, we illustrate an another proof of equation \eqref{eq:L(1/H)} for the support function $\varphi  = \langle F,  \nu \rangle$. As the soliton equation reads $- \frac{1}{H}= C {\varphi}$, we find 
$ {\nabla}_{\Sigma} H = H^2  {\nabla}_{\Sigma}  \left( - \frac{1}{H} \right) = C H^2 {\nabla}_{\Sigma} \phi$ and $H= - CH^2 {\varphi}$.  
The general formula deduced in, for instance, \cite[Proposition 6]{ADR2007}, \cite[page 291]{Huisken1990}, and  \cite[page 2]{PRS2005}, 
\[
 0 = \Delta_\Sigma   \langle F,  \nu \rangle  +     \langle F, {\nabla}_{\Sigma} H \rangle  +   {\vert A \vert}^{2}     \langle F,  \nu \rangle + H,
\]
can be rewritten as 
\[
 0 = \Delta_\Sigma  {\varphi} + CH^2\left\langle F, \nabla_{\Sigma}  {\varphi} \right\rangle + \left(\abs{A}^2 - CH^2\right)  {\varphi}.
 \]
\end{remark}


 \section{Main result and proof}
 \label{sec:main}

In this section, we prove that a complete immersed self-expander to the inverse mean curvature flow, which has one end asymptotic to a cylinder, or has two ends asymptotic to two coaxial cylinders, must be rotationally symmetric. Our proof requires the notion of asymptocity, which guarantees that the two kernel functions on the self-expander, induced by rotations and dilations, uniformly converge to corresponding constant kernel functions on the asymptotic cylinder(s), respectively. We shall introduce the following geometric definition of cylindrical asymptoticity, which implies the analytic uniform convergence of the kernel functions. 
  
\begin{definition}  \label{asmp}
Let $F_{\text{cyl}}(\omega, z) : \mathbb{S}^{n-1} \times \mathbb{R} \to \mathbb{R}^{n+1}$ be the embedding of a round cylinder $F_{\text{cyl}}(\omega, z) = ( \rho \, \omega, z)$, for some $\rho>0$. An end $E$ of an immersion $F: \Sigma \to \mathbb{R}^{n+1}$ is said to be asymptotic to this cylinder when the immersion is a normal graph over the cylinder $F_{\text{cyl}}$ on $E$, i.e. there exists a function $u : \mathbb{S}^{n-1} \times [z_0,\infty) \to \mathbb{R}$, for some $z_0 \in \mathbb{R}$, such that
\[
F(\omega, z) = F_{\text{cyl}}(\omega,z) + u(\omega, z) \nu_{\text{cyl}}(\omega,z),
\]
where $\nu_{\text{cyl}}$ is the unit normal of the round cylinder, and  the following growth assumptions hold:
\begin{enumerate}[label=(a\arabic*)]
\item \label{growth1}  $\displaystyle  \underset{z \to \infty} \lim \left( \, \underset{\omega \in  \mathbb{S}^{n-1}} \sup  \vert F(\omega, z) - F_{\text{cyl}}(\omega,z) \vert \right) = 0$,
\item \label{growth2}   $\displaystyle  \underset{z \to \infty} \lim \left( \, \underset{\omega \in  \mathbb{S}^{n-1}} \sup  \vert {\nu}_{F}(\omega, z) - {\nu}_{\text{cyl}}(\omega,z) \vert  \right) = 0$,
\item \label{growth3}  $\displaystyle  \underset{z \to \infty} \lim \left( \, \underset{\omega \in  \mathbb{S}^{n-1}} \sup  \, \vert  \langle  F(\omega, z)  ,  {\nu}_{F}(\omega, z) \rangle -  \langle F_{\text{cyl}}(\omega, z)  ,  {\nu}_{\text{cyl}}(\omega, z)  \rangle \vert \, \right) = 0$,
\end{enumerate}
where ${\nu}_{F}(\omega, z)$ denotes the induced unit normal of the self-expander $F\left(\Sigma\right)$.
\end{definition}

\begin{remark}  \label{GnA}
We note the uniform growth $|F_{\text{cyl}}(\omega,z)| = \mathcal{O}\left( \vert z \vert \right)$ as $z \to \infty$ in Defintion \ref{asmp}. If we strengthen the growth assumption \ref{growth2} by 
\begin{equation} \label{growth2str}
\underset{z \to \infty} \lim \left( \, \vert z \vert \, \underset{\theta \in  \mathbb{S}^{n-1}} \sup  \vert {\nu}_{F}(\omega, z) - {\nu}_{\text{cyl}}(\omega,z) \vert \, \right) = 0, 
\end{equation}
then, by observing the pointwise estimation, at each point $(\omega, z)$,
\begin{equation} \label{g2str}
\vert  \langle  F   ,  {\nu}_{F}  \rangle -  \langle F_{\text{cyl}}   ,  {\nu}_{\text{cyl}}   \rangle \vert  \leq 
\vert  \langle  F - F_{\text{cyl}}  ,  {\nu}_{F}  \rangle \vert +  \vert  \langle   F_{\text{cyl}} , {\nu}_{F} - {\nu}_{\text{cyl}}   \rangle \vert 
\leq \vert   F - F_{\text{cyl}}   \vert +  \vert     F_{\text{cyl}} \vert \cdot \vert {\nu}_{F} - {\nu}_{\text{cyl}} \vert,
\end{equation}
we find that  combining \ref{growth1}  and  (\ref{growth2str}) yields the third growth assumption \ref{growth3}. 
\end{remark}

\begin{remark}
Instead of the \textit{geometric} growth conditions in Defintion \ref{asmp}, we are  also able to impose that the normal distance function $u(\omega,z)$ defined on the round cylinder satisfies the \textit{analytic} growth conditions:
\begin{enumerate}[label=(a\arabic*')]

\item \label{growth1'} $\displaystyle \underset{z \to \infty} \lim \left( \, \sup_{\omega \in \mathbb{S}^{n-1}} |u(\omega,z)| \right) = 0$,

\item \label{growth2'} For any  first order differential operator $D$ on $\mathbb{S}^{n-1}$, $\displaystyle \underset{z \to \infty} \lim \left( \, \sup_{\omega \in \mathbb{S}^{n-1}} |Du (\omega,z)| \right) = 0$,

\item \label{growth3'} $\displaystyle \underset{z \to \infty} \lim \left( \, \sup_{\omega \in \mathbb{S}^{n-1}} |z \, u_z(\omega,z)| \right) = 0$.

\end{enumerate}

\end{remark}

 Now, we are ready to prove the main result.

\begin{theorem}[\textbf{Rotational symmetry of self-expanders with asymptotically cylindrical ends}] \label{thm:main}
Let $F:\Sigma^{n \geq 2} \to \mathbb{R}^{n+1}$ be a complete, immersed self-expander to the inverse mean curvature flow, and suppose either:
\begin{itemize}
\item $F(\Sigma)$ has only one end which is asymptotic, in the sense of Defiition~\ref{asmp}, to a round cylinder; or
\item $F(\Sigma)$ has only two ends which are asymptotic, in the sense of Defiition~\ref{asmp}, to two coaxial round cylinders.
\end{itemize}
Then, $F(\Sigma)$ must be rotationally symmetric with respect to the axis of the asymptotic cylinder(s).
\end{theorem}

\begin{proof}
For simplicity, we first present the proof of the case $n=2$ with one end. The argument can be adapted, \emph{mutatis mutandis}, to show the two-end and the higher dimension cases. 

Let $\widetilde{\Sigma}$ denote the asymptotic cylinder of radius $\rho>0$ in $\mathbb{R}^3$. Since the inverse mean curvature flow is invariant under rigid motions in ${\mathbb{R}}^{3}$, without loss of generality, 
we can take the $z$-axis as the axis of the rotational symmetry of $\widetilde{\Sigma}$.
Suppose that the self-expander $F: \Sigma \to \mathbb{R}^{n+1}$ has an end $E$ which is asymptotic to the cylinder
$\sqrt{\,x^2 +y^2\,}=\rho>0, \; z \in [z_0, +\infty)$ as in Definition \ref{asmp}.
On this cylinder $\widetilde{\Sigma}$, the rotation function and the support function are given by $\langle R_{\widetilde{\Sigma}}, \nu_{\widetilde{\Sigma}}\rangle \equiv 0$ and $\langle F_{\widetilde{\Sigma}}, \nu_{\widetilde{\Sigma}}\rangle \equiv -\rho$. As the end $E$ of  $F(\Sigma)$ is asymptotic to $\widetilde{\Sigma}$, by letting $z \to \infty$ on $E$, we find that $\langle R, \nu\rangle$ converges uniformly to $0$, and $\langle F, \nu \rangle$ converges uniformly to $-\rho$. (The proof of uniform convergence of $\langle R, \nu\rangle$ at infinity uses a similar estimation to (\ref{g2str}) of Remark \ref{GnA}.) 
Recall that the support function $\langle F, \nu\rangle$ does not vanish on a self-expander to IMCF. In particular, $\langle F, \nu\rangle$ is uniformly bounded away from zero on $\Sigma$, and so the quotient:
\[h := \frac{\langle R, \nu\rangle}{\langle F, \nu\rangle}\]
converges uniformly to $0$ as $z \to \infty$. Hence, even though $F(\Sigma)$ is non-compact, we find that the bounded 
function $|h|$ must attain a finite global maximum value on $\Sigma$.

To show that $F(\Sigma)$ is rotationally symmetric, we prove that $h$ and hence $\langle R, \nu \rangle$ vanish on $\Sigma$. Seeking a contradiction, suppose $h$ is non-zero somewhere on $\Sigma$. As $h$ converges uniformly to $0$ at infinity, the quotient function $h$ will achieve either a positive global maximum value or a negative global minimum value on the whole surface $\Sigma$.
 By applying the quotient identity
\[
  {\Delta}_{\Sigma} \left( \frac{f}{g} \right) = \frac{ g    {\Delta}_{\Sigma} f - f   {\Delta}_{\Sigma} g    }{   g^2 }   - \frac{2}{g}  \left\langle  {\nabla}_{\Sigma} g , 
  {\nabla}_{\Sigma}  \left( \frac{f}{g} \right) \right\rangle
\]
with $f = \langle R, \nu\rangle$ and $g = \langle F, \nu\rangle$, and combining with the two identities \eqref{eq:Lf_R} and \eqref{eq:L(1/H)} in Lemma \ref{kernel}, we can easily derive the equality 
\[
{\Delta}_{\Sigma} h + \left\langle C H^2 F + \frac{2}{\langle F, \nu\rangle} \nabla_{\Sigma} \langle F, \nu \rangle, \; {\nabla}_{\Sigma} h \right\rangle = 0
\]
on $\Sigma$. Since the quotient function $h : \Sigma \to \mathbb{R}$ achieves either a finite global minimum or maximum value, the strong maximum (minimum) principle guarantees that $h$ must be a non-zero constant function. This contradicts the behavior of $h$ as $z \to \infty$.
Hence, $h \equiv 0$ and $\langle R, \nu \rangle \equiv 0$  on $\Sigma$.

Since $\langle R, \nu \rangle$ vanishes on $\Sigma$, the rotation vector field $R=\left(-y,  x, 0 \right) \vert_{\Sigma}$ is 
tangent to $F(\Sigma)$ everywhere.  As is well-known (for instance, \cite[Section 2]{ALP1999}), one can easily check that, for 
any given point $\left( x_0, y_0, z_0 \right) \in F(\Sigma)$, the 
maximal integral curve of the vector field $R$ passing through  $ \left( x_0, y_0, z_0 \right)$ must be the horizontal circle having the radius $\sqrt{{x_{0}}^{2} +{y_{0}}^{2}} \geq 0$ and the center located in the $z$-axis. Hence, the self-expander is invariant under the rotations with respect to the $z$-axis. 

The two-end case can be proved in the same way. If $E_+$ and $E_-$ are the only two ends of $\Sigma$ which are asymptotic to two coaxial cylinders, we still have $\langle R, \nu\rangle$ converges uniformly to $0$ on both ends $E_+$ and $E_-$, and that $\langle F, \nu\rangle$ is uniformly bounded away from $0$ on $\Sigma$. The same argument above shows the quotient $h$ and hence $\langle R, \nu\rangle$ vanish on the whole $\Sigma$.

In higher dimensions, the group of rotations about the axis of the round cylinder $\widetilde{\Sigma}^n := \mathbb{S}^{n-1} \times \mathbb{R}$ in $\mathbb{R}^{n+1}$ is isomorphic to $\text{SO}(n)$. Consider \emph{any} smooth 1-parameter family of diffeomorphisms $\psi_t$ in the rotation group of $\widetilde{\Sigma}$, and let $R$ be the vector field generated by $\psi_t$, i.e. $\frac{\partial\psi_t}{\partial t} = R \circ \psi_t$. Clearly, we have $\langle R_{\widetilde{\Sigma}^n}, \nu_{\widetilde{\Sigma}^n} \rangle = 0$ for the cylinder $\widetilde{\Sigma}^n$. Then, we can adapt the same argument in the $n = 2$ case to show $\langle R, \nu\rangle \equiv 0$ on $\Sigma$. This implies that $\frac{\partial\psi_t}{\partial t}$  is tangent to $F(\Sigma)$ and hence $\psi_t(F(p)) \in F(\Sigma)$ whenever $p \in \Sigma$. Therefore, $\psi_t$ is also in the isometry group of $F(\Sigma)$. Now we have shown that the self-expander $\Sigma$ is invariant under \emph{any} rotation $\psi_t$ about the axis of the asymptotic cylinder(s), and hence the self-expander must be rotationally symmetric.
\end{proof}

\begin{remark}
Recently, I. Castro and A. M. Lerma \cite{CL2015} investigated Lagrangian homothetic solitons for the inverse mean 
curvature flow. It would be interesting to explore the rigidity of solitons of higher codimension. 
\end{remark}

\textbf{Acknowledgement.}  The first and third named authors would like to thank Christine Breiner for discussion on 
the construction of new self-expanders to the inverse mean curvature flow. The second named author would like to 
thank Peter McGrath for some enlightening discussion.

\end{document}